\documentclass[oneside,english]{amsart}
\usepackage[T1]{fontenc}
\usepackage[latin9]{inputenc}
\usepackage{amsthm}
\usepackage{amstext}
\usepackage{setspace}
\usepackage{amssymb}
\makeatletter
\numberwithin{equation}{section}
\numberwithin{figure}{section}
\theoremstyle{plain}
\newtheorem{thm}{Theorem}
  \theoremstyle{plain}
  \newtheorem{prop}[thm]{Proposition}
  \theoremstyle{plain}
  \newtheorem{cor}[thm]{Corollary}
  \newcounter{casectr}
  \newenvironment{caseenv}
  {\begin{list}{{\itshape\ Case} \arabic{casectr}.}{%
   \setlength{\leftmargin}{\labelwidth}
   \addtolength{\leftmargin}{\parskip}
   \setlength{\itemindent}{\listparindent}
   \setlength{\itemsep}{\medskipamount}
   \setlength{\topsep}{\itemsep}}
   \setcounter{casectr}{0}
   \usecounter{casectr}}
  {\end{list}}
  \theoremstyle{plain}
  \newtheorem*{thm*}{Theorem}

\makeatother

\usepackage{babel}

\begin{document}

\title{On the topology of invariant subspaces of a shift of higher multiplicity}

\author{Giorgi Shonia}
\begin{abstract}
Following Beurling's theorem and a study of the topology of invariant
subspaces by R. Douglas and C. Pearcy \cite{DouglasPearcyTopoInvSubspaces}
description of path connected components of invariant subspace lattice
for shift of multiplicity one has been given by R.Yang \cite{Yang1997CrossSec}.
This paper generalizes result to arbitrary finite multiplicity. We
show that there exists one to one correspondence between the invariant
subspace lattice of shift of arbitrary finite multiplicity and the
space of inner functions. 
\end{abstract}
\maketitle
\newpage{}

\tableofcontents{}

\newpage{}

\section{Introduction}

Beurling's theorem gave an analytic description of invariant subspace
for a shift operator of multiplicity one. Later this result has been
generalized to shifts of arbitrary multiplicity. Through the study
of the topology of invariant subspace lattice Douglas and Pearcy have
proposed a problem of path connected component of the invariant subspace
lattice. Yang studied the relationship between one parameter families
of invariant subspaces of shift of multiplicity one and their respective
analytic description. Present paper generalizes the result to arbitrary
finite multiplicity.

\subsection{Shift operator and Beurling's theorem}

Complete analytic description of invariant subspaces of shift operator
of multiplicity one was given by Beurling \cite{Beurling1948}. We
bring main highlights of the theory (cf. Rosenthal and Radjavi \cite{Rosenthal-Radjavi:2003}). 

Let $H$ be a separable, complex, infinite-dimensional Hilbert space
with orthonormal basis $\{e_{n}\}_{n\in\mathbb{N}}$. Let $B(H)$
denote collection of all bounded linear operators on $H$. For an
operator $T\in B(H)$ a (closed) subspace $M\subset H$ is invariant
if $TM\subset M$ and reducing if in addition $TM^{\perp}\subset M^{\perp}$.
For given operator $T$ collection of all invariant subspaces ordered
by inclusion form a lattice denoted $\textrm{Lat}(T)$. 

A unilateral forward shift of multiplicity one is an isometry operator
$S\in B(H)$, whose action on Hilbert space basis vectors $\left\{ e_{n}\right\} _{n\in\mathbb{N}}$
is a shift forward $S(e_{n})=e_{n+1}$. If instead we enumerate the
basis of $H$ as $\{e_{n}\}_{n\in\mathbb{Z}}$ then isometry $U\in B(H)$
acting on $H$ by $U(e_{n})=e_{n+1}$ is called a bilateral forward
shift of multiplicity one. $U$ is a surjective isometry, $S$ is
an isometry, but not surjective. Both operators have adjoints, backwards
shifts of respective basis.

Description of invariant subspaces for shift operators is based on
their analytic interpretation, where Hilbert space is identified with
$L^{2}\left(\mathbb{T}\right)$, square integrable functions on the
circle (with normalized arc-length measure). Collection $\{z^{n}\}_{n\in\mathbb{Z}}$
where $z^{n}$ is a polynomial on $\mathbb{T}$ is a Hilbert basis
and bilateral shift is acting as operator $M_{z}$, multiplication
by $z$. For unilateral shift we have to consider Hardy space $H^{2}=\bigvee_{n\in\mathbb{N}}z^{n}$,
a subspace of $L^{2}$ restricted to the span of positive powers.
Unilateral shift is identified with $M_{z\big|H^{2}}$. With this
interpretation we can now formulate Beurling's theorem:
\begin{thm}
(Beurling): A subspace $M\subseteq L^{2}(\mathbb{T})$ is invariant
and non-reducing for the bilateral shift $U$ if and only if there
exists a measurable function $G(z)$ on $\mathbb{T}$ with $|G(z)|=1$
ae, such that $M=GH^{2}$. 

Furthermore two spaces of the form $G_{1}H^{2}$ and $G_{2}H^{2}$with
$|G_{1}|=|G_{2}|=1$ $z$-ae are equal if and only if the ratio $\frac{G_{1}}{G_{2}}$
is constant $z$-ae. 
\end{thm}
The functions $G\in H^{2}$ with $|G(z)|\stackrel{\textrm{ae}}{=}1$
are called an inner functions.

\subsection{Shifts of higher multiplicity}

Let $K$ be a separable Hilbert space of dimension $n=\dim(K)$. For
the purpose of construction of generalization of Beurling's theorem
dimension can be countable, though the main result of present paper
is only covering arbitrary finite dimension. Let $H=\bigoplus_{i\in\mathbb{Z}}K_{i}$
be an orthogonal sum of infinite copies of $K$, with the derived
summary inner product $\langle\sum v_{i};\sum w_{i}\rangle_{H}=\sum\langle v_{i};w_{i}\rangle_{K_{i}}$.
\emph{Bilateral} shift of multiplicity $n$ is an unitary operator
$U$ which is shifting each copy in the sum one position forward,
$UK_{i}=K_{i+1}$. The restriction of this on space $\bigoplus_{i\in\mathbb{N}}K_{i}$
is a \emph{unilateral} shift of multiplicity $n$. Similar to the
shift of multiplicity one, shift of higher multiplicity also has analytic
interpretation and the description of its invariant subspace can also
be generalized.

For analytic interpretation consider $K$-valued functions from unit
circle\linebreak{}
 $f:\mathbb{T}\rightarrow K$. Define $f$ to be measurable,
if $z\rightarrow\langle f(z),x\rangle_{K}$ is measurable for all
$x\in K$. It is not hard to show that function $z\rightarrow\|f(z)\|$
is also measurable and also for two $K$ valued measurable functions
$f,\, g$ the inner product $z\rightarrow\langle f(z),g(z)\rangle$
is measurable as well. This defines a Hilbert space $L^{2}(K)$ with
inner product $\langle f,g\rangle_{L^{2}(K)}=\int_{\mathbb{T}}\langle f(z),g(z)\rangle_{K}d\mu$.
Completeness of this space easily follows from the representation
of its elements with Fourier series. Every $f\in L^{2}\left(K\right)$
can be uniquely identified with series, $f=\sum_{i\in\mathbb{Z}}x_{i}e_{i}$
with $x_{i}\in K$ and this representation has the property that $\langle f(z),x\rangle\stackrel{z-ae}{=}\sum\langle x_{i},x\rangle z^{i}$
is true for almost everywhere for $z\in\mathbb{T}$ and for all $x\in K$.
It is also easy to show that for $f=\sum_{i\in\mathbb{Z}}x_{i}e_{i}$
we have the familiar formula $\|f\|^{2}=\sum_{i\in\mathbb{Z}}\|x_{i}\|^{2}$. 

As in one dimensional case, analytic interpretation of a bilateral
shift of multiplicity $n$ is a multiplication by $z$ on space $L^{2}(K)$,
operator $M_{z}$. Unilateral shift of multiplicity $n$ is an operator
$S:=U_{|H^{2}(K)}$ where the space \[
H^{2}(K):=\{f\in L^{2}(K)\,|\, f=\sum_{n\in\mathbb{N}}x_{i}e_{i}\}\]
 is subspace of $L^{2}(K)$ with coefficients of negative powers being
zero.

To generalize the notion of the inner functions we will consider operator
valued functions defined on the unit circle $\mathbb{T}$. A function
$F:\mathbb{T}\rightarrow B(K)$ is said to be measurable if for all
$x\in K$ the $K$-valued function $z\rightarrow F(z)x$ is measurable
(as previously defined). Let $\|F\|_{\infty}:=\textrm{esssup}_{z\in\mathbb{T}}\|F(z)\|_{B(K)}$.
When $\|F\|_{\infty}<\infty$ we can define an operator $M_{F}$ on
$L^{2}(K)$ by $(M_{F}f)(z)=F(z)f(z)$. Let \[
L^{\infty}(\mathbb{T},B(K)):=\{F:\mathbb{T}\rightarrow B(K)\,|\,\|F\|_{\infty}<\infty\}\]
 be a collection of all bounded measurable operator valued functions
from the unit circle and let \[
L^{\infty}(B(K)):=\{M_{F}\,|\, F\in L^{\infty}(\mathbb{T},B(K))\}\]
 be a collection of respective operators on $L^{2}(K)$. It is easy
to show that map \linebreak{}
$F\rightarrow M_{F}$ is an algebra isomorphism which
respects adjoint operation. As a consequence $M_{F}$ is unitary iff
$F$ is $z\in\mathbb{T}$ ae-unitary. We define a collection of analytic
elements of $L^{\infty}(\mathbb{T},B(K))$, described by \[
H^{\infty}(\mathbb{T},B(K)):=\{F\in L^{\infty}(\mathbb{T},B(K))\,|\, H^{2}(K)\in\textrm{Lat}(M_{F})\},\]
which respectively defines class of operators \[
H^{\infty}(B(K)):=\{M_{F}\,|\, F\in H^{\infty}(\mathbb{T},B(K))\}.\]
Operators in $H^{\infty}(B(K))$ have power series representation
which, though not crucial for our exposition, is helpful for an additional
intuitive perspective on the argument (cf Sz-Nagy and Foais \cite{BelaSz.-Nagy:2008}).
Consider an operator $G_{i}\in B(N;K)$ for $i\in\mathbb{N}$ and
let $G(z)=\sum_{i\in\mathbb{N}}z^{i}G_{i}$ where series converge
(weakly, strongly, in norm, all are the same for separable Hilbert
space) for all $z\in D$ in the open unit disk. Such series with a
condition that $\|G(z)\|\leq C$ for all $z\in D$ is called a bounded
analytic function. Each series can be associated with the operator
$(N,K,G)$ on $H^{2}(K)$ through strong limit $G(e^{t\theta})=\lim_{z\rightarrow e^{i\theta}}G(z)$,
where convergence is along non tangential path. This operator $(N,K,G)$
is the same as multiplication operator previously defined based on
operator valued functions, with initial space $H^{2}(N)$ and final
space $H^{2}(K)$, both definitions forming collection $H^{\infty}(B(K))$.
Such operator $(N,K,G)$ (or alternatively $M_{G}$) is an isometry
on $H^{2}(K)$ if and only if $G(e^{t\theta})\in B(N,K)$ is isometry
ae-$\theta$, in which case it is called inner function and has initial
space $H^{2}(N)$ and final space $H^{2}(K)$. In our main result,
we will identify $N$ with $\mathbb{C}^{m}$, $K$ with $\mathbb{C}^{n}$
and consider $G$ to be an isometry from $H^{2}(\mathbb{C}^{m})$
to $H^{2}(\mathbb{C}^{n})$. 

Generalization of Beurling's theorem to higher dimensions also involves
Wold decomposition\emph{:}
\begin{prop}
\emph{(Wold decomposition)}: Every invariant subspace $M$ of the
bilateral shift $U$ has a unique decomposition of the form $M=M_{1}\oplus M_{2}$
where $M_{1}$ is a reducing subspace for $U$, $M_{2}$ is an invariant
subspace of $U$ and $\bigcap_{i\in\mathbb{N}}U^{i}M_{2}=\{0\}$
\end{prop}
With this now we are able to formulate generalization of Beurling's
theorem:
\begin{thm}
\emph{(Generalization of Beurling):} A subspace $M$ of a Hilbert
space\linebreak{}
$H=\bigoplus_{i\in\mathbb{Z}}K_{i}$ (where $K_{i}$
is a Hilbert space) is an invariant subspace of the bilateral shift
$U$ if and only if $M=M_{1}\oplus M_{V}H^{2}(N)$ where $M_{1}$
reduces $U$, $N\subseteq K$ is a subspace, $V\in L^{\infty}(\mathbb{T},B(K))$
is a $z$-ae partial isometry on $K$ with initial space $N$.

\emph{Further: }Subspace\emph{ }$M_{1}$ is uniquely determined by
subspace $M$; and 

If $V_{1}\in L^{\infty}(\mathbb{T},B(K))$ is a partial isometry with
an initial space $N_{1}$ and\linebreak{}
$M_{V}H^{2}(N)=M_{V_{1}}H^{2}(N_{1})$, then there is
a partial isometry $W$ with an initial space $N_{1}$ and final space
$N$ such that $M_{V_{1}}=M_{V}W$.
\end{thm}
Since unilateral shift has no reducing subspace, we also have following
corollary:
\begin{cor}
A subspace $M$ of a Hilbert space $H=\bigoplus_{i\in\mathbb{Z}}K_{i}$
(where $K_{i}$ is a Hilbert space) is an invariant subspace of the
unilateral shift $S$ if and only if\linebreak{}
$M=M_{G}H^{2}(N)$ where $N$ is a subspace of $K$ and
$G\in H^{\infty}(\mathbb{T},B(K))$ is an operator valued function
such that $G(z)$ is $z-ae$ partial isometry with initial space $N$. 
\end{cor}
There is also a uniqueness result similar to the bilateral shift.

At this point we are ready to formulate the main result of the present
paper. 
\begin{thm}
\label{thm:main result}Let $H^{2}(\mathbb{C}^{n})$ be Hardy space
for some finite $n$ and let $\{p_{t}\}_{t\in[0,1]}$ be a family
of orthogonal projections such that:

1. Family is continuous in norm, $\lim_{t\rightarrow\tau}\|p_{t}-p_{\tau}\|=0$;

2. Each projection $p_{t}H^{2}(\mathbb{C}^{n})$ produces an invariant
subspace of a unilateral shift operator (of multiplicity $n$). 

Then there exists integer $m\leq n$ and a choice of a family of inner
functions $G_{t}\in H^{\infty}(\mathbb{T},B(\mathbb{C}^{m},\,\mathbb{C}^{n}))$
for $t\in[0,1]$ such that:

1. $p_{t}H_{\mathbb{C}^{n}}^{2}=G_{t}H_{\mathbb{C}^{m}}^{2}$; and

2. $\{G_{t}\}_{t\in[0,1]}$ are sup-norm continuous,\\
$\textrm{esssup}_{\theta}\|G_{t}(e^{i\theta})-G_{\tau}(e^{i\theta})\|_{B(\mathbb{C}^{m};\mathbb{C}^{n})}\stackrel{t\rightarrow\tau}{=}0$
\end{thm}
Present paper covers above result for an arbitrary finite $n$. In
section two we demonstrates continuity in $L^{2}$norm, third section
demonstrates continuity in sup-norm.

\subsection{The dimension of the wandering space.}

Given a family of norm continuous projections the first question is
whether the dimensions of respective inner functions, dimension of
the wandering space, of initial space $N$ in Beurling's theorem is
consistent across the family. This amounts to the existence of an
integer $m$ in the theorem \emph{\ref{thm:main result}}. The following
proposition addresses this question affirmatively.
\begin{prop}
Consider Hardy space $H^{2}(\mathbb{C}^{n})$ for some finite $n$
and a family of projections $\{p_{t}\}_{t\in[0,1]}$ such that:

1. Family is continuous in norm, $\lim_{t\rightarrow\tau}\|p_{t}-p_{\tau}\|=0$;

2. Each projection $p_{t}H^{2}(\mathbb{C}^{n})$ produces an invariant
subspace of a unilateral shift operator (of multiplicity $n$). 

Let $W_{t}=\left(I-S(S_{|p_{t}H^{2}})^{*}\right)p_{t}H_{\mathbb{C}^{m}}^{2}$
be respective wandering space.

Then $\textrm{dim}(W_{t})=\textrm{const}$ for all $t\in[0,1]$.\end{prop}
\begin{proof}
We will use the fact, that if distance between two projections is
less then one, then they have the same rank. \[
\lim_{t\rightarrow\tau}\|\left(I-S(S_{|p_{t}H^{2}})^{*}\right)p_{t}-\left(I-S(S_{|p_{\tau}H^{2}})^{*}\right)p_{\tau}\|\leq\]
\[
\leq\|S\|\lim_{t\rightarrow\tau}\left\Vert (S_{|p_{t}H^{2}})^{*}p_{t}-(S_{|p_{\tau}H^{2}})^{*}p_{\tau}\right\Vert +\lim_{t\rightarrow\tau}\|p_{t}-p_{\tau}\|=0<1\]

that is (eventually) $W_{t}$ and $W_{\tau}$ have the same dimension
$m$, thus rank is constant on the whole $t\in[0,1]$ interval.
\end{proof}

\section{continuity in $L^{2}$ norm}

Beurling's theorem only defines inner function $G_{t}$ up to a unitary
rotation (for each $t\in[0,1]$). In order to construct continuous
path $G_{t}$ for $t\in[0,1]$ we need a consistent method of selecting
unique $G_{t}$ for each $t$. 

For the clarity of exposition, we will first consider a basic case
$m=n$ and then the general case $m\leq n$, whose treatment is only
different from that of the basic case by an additional care which
is needed for pinning down an unique $G_{t}$ for each $t$ from a
unitarily equivalent family. 
\begin{prop}
Consider Hardy space $H^{2}(\mathbb{C}^{n})$ for some finite $n$
and a family of projections $\{p_{t}\}_{t\in[0,1]}$ such that:

1. Family is continuous in norm, $\lim_{t\rightarrow\tau}\|p_{t}-p_{\tau}\|=0$;

2. Each projection $p_{t}H^{2}(\mathbb{C}^{n})$ produces an invariant
subspace of a unilateral shift operator (of multiplicity $n$). 

Consider analytical description of this subspace $G_{t}H^{2}(\mathbb{C}^{m})$
for some inner function $G_{t}$ from $H^{2}(\mathbb{C}^{m})$ to
$H^{2}(\mathbb{C}^{n})$ where $m\leq n$ and $p_{t}H^{2}(\mathbb{C}^{n})=G_{t}H^{2}(\mathbb{C}^{m})$

Then the family $\{G_{t}\}_{t\in[0,1]}$ can be chosen to be $L^{2}$
continuous, specifically \[
\int_{\theta}\|G_{t}(e^{i\theta})-G_{\tau}(e^{i\theta})\|_{B(\mathbb{C}^{m};\mathbb{C}^{n})}^{2}d\mu\stackrel{t\rightarrow\tau}{=}0\]

\end{prop}
\emph{Proof: }First we consider the case, when the dimension $n$
from the construction of Hardy space $H^{2}(\mathbb{C}^{n})$ is the
same as the dimension of the wandering space $W$.
\begin{caseenv}
\item If $m=n$ we can go to a point $\lambda\in D$ in the unit disk \cite{Bercovici1990}.
We use a reproducing kernel in order to relate projection $p_{t}=\textrm{Proj}_{G_{t}H^{2}}$
to the inner function $G_{t}$. We consider an inner product $\left\langle p_{t}\frac{f}{1-\overline{\lambda}z};\, z^{k}G_{t}(z)u\right\rangle $
where $f\in\mathbb{C}^{n}$ and $u\in\mathbb{C}^{m}$ are arbitrary
vectors. We get\[
\left\langle p_{t}\frac{f}{1-\overline{\lambda}z};\, z^{k}G_{t}(z)u\right\rangle =\left\langle \frac{f}{1-\overline{\lambda}z};\, z^{k}G_{t}(z)u\right\rangle =\left\langle f;\,\lambda^{k}G_{t}(\lambda)u\right\rangle =\]
\[
=\left\langle G_{t}(\lambda)^{*}f;\,\lambda^{k}u\right\rangle =\left\langle \frac{G_{t}(\lambda)^{*}f}{1-\overline{\lambda}z};\, z^{k}u\right\rangle =\left\langle \frac{G_{t}(z)_{t}G(\lambda)^{*}f}{1-\overline{\lambda}z};\, z^{k}G(z)u\right\rangle \]
At this point we can replace $f$ with $e_{1},\ldots e_{n}\in\mathbb{C}^{n}$
to get \begin{equation}
\left(p_{t}\frac{e_{1}}{1-\overline{\lambda}z},\, p_{t}\frac{e_{2}}{1-\overline{\lambda}z}\ldots p_{t}\frac{e_{n}}{1-\overline{\lambda}z}\right)=\frac{G_{t}(z)G_{t}(\lambda)^{*}}{1-\overline{\lambda}z}\label{eq:lambda}\end{equation}
Denote \[
\Lambda_{t}(z):=\left(p_{t}\frac{e_{1}}{1-\overline{\lambda}z},\, p_{t}\frac{e_{2}}{1-\overline{\lambda}z}\ldots p_{t}\frac{e_{n}}{1-\overline{\lambda}z}\right).\]
Taking $z=\lambda$ we get \[
G_{t}(\lambda)G_{t}(\lambda)^{*}=(1-|\lambda|^{2})\Lambda_{t}(\lambda).\]
Since the right hand side is expressed in terms of $p_{t}$, it is
a continuous function in $t\in[0,\,1]$. The left hand side is also
continuous as a function of $\lambda\in D$. Hence $\textrm{rank}(G_{t}(\lambda))=\textrm{rank}(G_{t}(\lambda)G_{t}(\lambda)^{*})$
is lower-semi-continuous in $\lambda$ and $t$. If $G_{t_{0}}(\lambda_{t_{0}})$
has a full rank $n$ for some $\lambda_{t_{0}}$ then by lower-semi-continuity
$G_{t}(\lambda_{t_{0}})$ must also have same rank $n$ for $t$ close
to $t_{0}$. We claim that for every $t\in[0,\,1]$ there exists such
$\lambda_{t}$ where $G_{t}(\lambda_{t})$ has a full rank $n$. For
if no such $\lambda_{t}$ exists for some $t$, we must have $\textrm{det}G_{t}(\lambda)=0$
for every $\lambda\in D$ and therefore $\textrm{det}G_{t}(z)=0$
a.e. for $z\in\mathbb{T}$. This however is impossible since $G_{t}(z)$
is an isometry and $|\textrm{det}G_{t}(z)|=1$ for a.e. $z\in\mathbb{T}$.
The preceding remark and compactness of $[0,\,1]$ produces a collection
of points $\lambda_{1},\lambda_{2}\ldots\lambda_{N}\in D$ and a finite
open cover $\{(a_{k},\, b_{k})\}_{k=1}^{N}$ of interval $[0,\,1]$
such that $G_{t}(\lambda_{k})$ is invertible for $t\in(a_{k},\, b_{k})$
for $k=1,2\ldots N$.\\
We set \[
G_{k,t}(\lambda_{k})=\sqrt{(1-|\lambda_{k}|^{2})\Lambda_{t}(\lambda_{k})}\]
 for $k=1,2\ldots N$ and $t\in(a_{k},\, b_{k})$. Within each interval
$(a_{k},\, b_{k})$ this definition identifies one choice from a family
of inner functions which so far was only defined up to an unitary
rotation by Beurling's theorem. This choice is also continuous, $G_{k,t}(\lambda_{k})$
is continuous in $t$. Having picked $G_{k,t}$ at a single point
$\lambda_{k}\in D$ we consistently pick it from the unitarily equivalent
family on the whole unit circle: $G_{k,t}(z)=(1-\overline{\lambda}_{k}z)\Lambda_{t}(z)G_{t}(\lambda_{k})^{-1}$.
Thus defined $G_{k,t}$ is $L^{2}$ continuous in $t$ for $t\in(a_{k},\, b_{k})$
and $G_{k,t}H^{2}(\mathbb{C}^{n})=p_{t}H^{2}(\mathbb{C}^{n})$.\\
Finally we patch up $G_{k,t}$ functions defined in intervals $t\in(a_{k},\, b_{k})$,\linebreak{}
$k=1,2\ldots N$ into a single $\tilde{G_{t}}$, adjusting
by rotation so that definition on every subsequent interval to be
continuation in $t$ of previous. Starting from $t=0$ we put $\tilde{G_{t}}=G_{0,t},$
for $t\in[a_{0}=0;b_{0}]$. Next, since $b_{0}\in(a_{1};b_{1})$ we
have $G_{1,b_{0}}U_{1}=\tilde{G}_{b_{0}}$, where $U_{1}$ is an unitary
constant. We put $\tilde{G_{t}}=G_{1,t}U_{1}$ for $t\in(b_{0};b_{1}]$.
This extends $\tilde{G}_{t}$ to the next interval continuously. Carrying
on the process in the same manner for all covering intervals $\{(a_{k},\, b_{k})\}_{k=1}^{N}$
we get single inner function $\tilde{G}_{t}$ which is $L^{2}$ continuous
in $t$ for $t\in[0,\,1]$ and produces desired invariant subspaces:
$\tilde{G}_{t}H^{2}(\mathbb{C}^{n})=p_{t}H^{2}(\mathbb{C}^{n})$
\item If $m<n$ the calculation leading to the equation \ref{eq:lambda}
is still valid, and a function $d(\lambda)=\textrm{rank}(G_{t}(\lambda))=\textrm{rank}(G_{t}(\lambda)G_{t}(\lambda)^{*})$
is lower-semi-continuous in $\lambda\in D$. Observe that for $\theta-$a.e.
$\textrm{rank}(G_{t}(e^{i\theta}))=m$. Since $G_{t}(e^{i\theta})$
is a strong limit of $G_{t}(re^{i\theta})$ as $r\rightarrow1$ (and
matrix being of finite dimension), we get that for all $t\in[0,1]$
there exists $\lambda_{t}\in D$ such that $G_{t}(\lambda)$ has rank
$m$. If $G_{t_{0}}(\lambda_{t_{0}})$ has rank $m$ for some $\lambda_{t_{0}}$
then by lower-semi-continuity $G_{t}(\lambda_{t_{0}})$ must also
have same rank $m$ for $t$ close to $t_{0}$. With the same argument
as in the previous case we get a collection of points $\lambda_{1},\lambda_{2}\ldots\lambda_{N}\in D$
and a finite open cover $\{(a_{k},\, b_{k})\}_{k=1}^{N}$ of interval
$[0,\,1]$ such that $G_{t}(\lambda_{k})$ is one-to-one for $t\in(a_{k},\, b_{k})$
for $k=1,2\ldots N$. The only instance where argument of previous
case fails is obtaining $G_{k,t}(\lambda_{k})$ from $G_{k,t}(\lambda_{k})=\sqrt{(1-|\lambda_{k}|^{2})\Lambda_{t}(\lambda_{k})}$,
in present case we have to pick a unique representative $G_{k,t}$
from a family of unitarily equivalent ones (produced by Beurling's
theorem) in a different manner. Provided we succeed in such selection,
the rest of the argument can be finished as before, in $m=n$ case.
Having defined $G_{k,t}$ at a single point $\lambda_{k}\in D$ we
define it accordingly on the whole unit circle: $G_{k,t}(z)=(1-\overline{\lambda}_{k}z)\Lambda_{t}(z)\left(G_{t}(\lambda)_{\big|rang(G^{*})}^{*}\right)^{-1}$
where $\left(G_{t}(\lambda)_{\big|rang(G^{*})}^{*}\right)^{-1}$ is
$n\times m$ matrix acting as inverse of $G_{t}(\lambda)^{*}$ on
the $\textrm{range}(G_{t}^{*})$. Thus defined $G_{k,t}$ is $L^{2}$
continuous in $t$ for $t\in(a_{k},\, b_{k})$ and $G_{k,t}H^{2}(\mathbb{C}^{m})=p_{t}H^{2}(\mathbb{C}^{n})$.
Patching between intervals $(a_{k},\, b_{k})$ for $k=1,2\ldots N$
is again done similar to preceding case.\\
In order to fix unitary rotation we look at operator $A_{t}(\lambda)=\sqrt{(1-|\lambda|^{2})\Lambda_{t}(\lambda)}$
and get $G_{t}(\lambda)G_{t}(\lambda)^{*}=A_{t}^{2}$, that is $G_{t}(\lambda)^{*}=U_{t}A_{t}$
where $U_{t}$ is a surjective partial isometry. Fixing $G_{t}$ is
equivalent to fixing $U_{t}$ which is equivalent to picking an orthonormal
basis for range -$\ker^{\perp}$ space of the positive matrix $A_{t}$.
\\
If all (nonzero) eigenvalues $\mu_{1}(t)>\mu_{2}(t)>\ldots\mu_{m}(t)$
of $A_{t}$ are distinct, then we can pick an orthonormal basis $\{v_{1}(t),\, v_{2}(t)\ldots v_{m}(t)\}$
of eigenvectors according to the descending order of eigenvalues.
If eigenvalues $\{\mu_{k}(t)\}_{t=1}^{m}$ of $A_{t}$ remain distinct
for all $t\in[0,1]$ then for all $t$ we have an orthonormal basis
$\{v_{k}(t)\}_{t=1}^{m}$ of range -$\ker^{\perp}$ space of the positive
matrix $A_{t}$, this basis is continuous in $t$ and proposition
is proved.\\
For illustration how to relax this last restriction lets assume
that eigenvalues of $A_{t}$ are all distinct except on interval $t\in[\tau_{1},\,\tau_{2}]$
where certain eigenvalue $\mu_{i}(t)=\mu_{j}(t)$ has multiplicity
$2$. Then we keep track of {}``entry'' $\mu_{i}(\tau_{1})$, $\mu_{j}(\tau_{1})$
and {}``exit'' $\mu_{i}(\tau_{2})$, $\mu_{j}(\tau_{2})$ directions.
We know two vectors entering into two dimensional eigenspace $\textrm{Span}\{\mu_{i}(\tau_{1}),\mu_{j}(\tau_{1})\}$
at $t=\tau_{1}$ and exiting from $\textrm{Span}\{\mu_{i}(\tau_{2}),\mu_{j}(\tau_{2})\}$
at $t=\tau_{2}$. In between of two endpoints we pick linear interpolation
(or any $t$-continuous transition) between entry and exit directions
(keeping two dimensional plane $\textrm{Span}\{\mu_{i}(t),\mu_{j}(t)\}$
under the same orientation). Thus we preserve continuity of $U_{t}$.
If it happens that $\tau_{1}=0$ we are at liberty to pick any two
orthogonal directions in eigenspace, since we are not limited with
left-continuity at $t=0$ (respectively right-continuity if $\tau_{2}=1$).
In this way any combination of multiplicity can be handled. By handling
new arrivals and departures to eigenspace in some consistent way (one
choice could be last in, first out) continuity of $U_{t}$ will be
maintained. $\Box$
\end{caseenv}

\section{$L^{\infty}$ continuity}

Using results of previous section we choose $G_{t}$ to be $L^{2}$
continuous for $t\in[0,\,1]$ and we argue that in fact we have $L^{\infty}$
continuity. We go inside the disk at some point $(rz)$ and use reproducing
kernel for evaluation at that point. We fix some $r\in(0,1)$ and
kernel $k_{z}(w)=\frac{1}{1-r\bar{z}w}$.

For all $f\in H_{\mathbb{C}^{n}}^{2}$\begin{eqnarray*}
(M_{G}^{*}f)(rz) & = & \left\langle M_{G}^{*}f;k_{z}e_{1}\right\rangle e_{1}+\left\langle M_{G}^{*}f;k_{z}e_{2}\right\rangle e_{2}+\ldots+\left\langle M_{G}^{*}f;k_{z}e_{m}\right\rangle e_{m}\\
 & = & \left\langle f;M_{G}(k_{z}e_{1})\right\rangle e_{1}+\left\langle f;M_{G}(k_{z}e_{2})\right\rangle e_{2}+\ldots+\left\langle f;M_{G}(k_{z}e_{m})\right\rangle e_{m}\end{eqnarray*}

\newpage{}\begin{eqnarray*}
 & = & \left\langle f;k_{z}g_{\cdot1}\right\rangle e_{1}+\left\langle f;k_{z}g_{\cdot2}\right\rangle e_{2}+\ldots+\left\langle f;k_{z}g_{\cdot m}\right\rangle e_{m}\\
 & = & \left(\begin{array}{c}
\int f\cdot\overline{k_{z}g_{\cdot1}}\\
\int f\cdot\overline{k_{z}g_{\cdot2}}\\
\cdots\\
\int f\cdot\overline{k_{z}g_{\cdot m}}\end{array}\right)=\int_{\mathbb{T}}\overline{k_{z}(w)}G(w)^{*}f(w)dw\end{eqnarray*}

Respectively \[
(M_{G}M_{G}^{*}f)(rz)=\int_{\mathbb{T}}\overline{k_{z}(w)}G(rz)G(w)^{*}f(w)dw=\int_{\mathbb{T}}K^{*}(z,w)f(w)dw\]
 where $K_{t}(z,w):=\frac{1}{1-r\bar{z}w}G_{t}(w)G_{t}(rz)^{*}$.
We fix some $s,t\in[0,1]$ and put\linebreak{}
$F(z,w)=K_{t}(z,w)-K_{s}(z,w)$, for $(z,w)\in\mathbb{T}\times\mathbb{T}$. 
\begin{prop}
Consider function $K_{t}(z,w)=\frac{1}{1-r\bar{z}w}G_{t}(w)G_{t}(rz)^{*}$
defined for $(z,w)\in\mathbb{T}\times\mathbb{T}$ where $G_{t}$ is
an inner function. For $0<r<1$ consider\linebreak{}
$F(z,w)=K_{t}(z,w)-K_{s}(z,w)$. Let $\|F\|_{\infty}$
be essential supremum taken against planar measure over $(z,w)\in\mathbb{T}\times\mathbb{T}$.
Given $z$ let $\|F(z,\cdot)\|_{\infty}$ be $w$-pointwise essential
supremum of matrix norm of $F(z,w)$.

Then $\|F\|_{\infty}=\sup_{z\in\mathbb{T}}\|F(z,\cdot)\|_{\infty}$ \end{prop}
\begin{proof}
The proof in higher dimension is similar to a scalar case. One direction
is straightforward. If $\|F(z,w)\|_{B(\mathbb{C}^{n})}>\|F\|_{\infty}-\epsilon$
on some set $A$ of positive planar measure, then the same is true
on some one dimensional section $A_{z_{0}}$ of positive measure,
that is $\|F(z_{0},\cdot)\|_{\infty}>\|F\|_{\infty}-\epsilon$ , giving
$\sup_{z\in\mathbb{T}}\|F(z,\cdot)\|_{\infty}>\|F\|_{\infty}$.

On the other hand \begin{eqnarray*}
\|F(z,w)-F(z',w)\|_{B(\mathbb{C}^{n})} & = & \left\Vert \frac{1}{1-r\bar{z}w}\big(G_{t}(w)G_{t}(rz)^{*}-G_{s}(w)G_{s}(rz)^{*}\big)-\right.\\
 &  & \left.-\frac{1}{1-r\bar{z}'w}\big(G_{t}(w)G_{t}(rz')^{*}-G_{s}(w)G_{s}(rz')^{*}\big)\right\Vert \end{eqnarray*}

\newpage{}\begin{eqnarray*}
 & \leq & \frac{r|z-z'|}{(1-r)^{2}}\left\Vert G_{t}(w)G_{t}(rz')^{*}-G_{s}(w)G_{s}(rz')^{*}\right\Vert +\\
 &  & +\frac{1}{1-r}\left\Vert G_{t}(w)\big(G_{t}(rz)^{*}-G_{t}(rz')^{*}\big)-G_{s}(w)\big(G_{s}(rz)^{*}-G_{s}(rz')^{*}\big)\right\Vert \\
 & \leq & \frac{2r}{(1-r)^{2}}|z-z'|\\
 &  & +\frac{1}{1-r}\left(\|G_{t}(rz)^{*}-G_{t}(rz')^{*}\|_{B(\mathbb{C}^{n})}+\|G_{s}(rz)^{*}-G_{s}(rz')^{*}\|_{B(\mathbb{C}^{n})}\right)\end{eqnarray*}

and since $rz$ is inside the unit disk, $F(z,w)$ is uniformly (with
regard to $w$) continuous in $z$. Then $\|F(z,\cdot)\|_{\infty}$
is continuous in $z$. If there is a significant one dimensional section
where supremum is approached in $w$, then that value will not be
changed much while perturbing $z$ variable. Thus there is a set of
positive planar measure on $\mathbb{T}\times\mathbb{T}$ where $\|F\|_{\infty}>\sup_{z\in\mathbb{T}}\|F(z,\cdot)\|_{\infty}-\epsilon$\end{proof}
\begin{prop}
Let $p_{t}$ and $p_{s}$ be orthogonal projections in $H^{2}(\mathbb{C}^{n})$
on invariant subspaces of a unilateral shift of multiplicity $n$.
Let $G_{t}$ and $G_{s}$ be an inner function from $H^{2}(\mathbb{C}^{m})$
to $H^{2}(\mathbb{C}^{n})$ such that \begin{eqnarray*}
p_{t}H^{2}(\mathbb{C}^{n}) & = & G_{t}H^{2}(\mathbb{C}^{m})\\
p_{s}H^{2}(\mathbb{C}^{n}) & = & G_{s}H^{2}(\mathbb{C}^{m})\end{eqnarray*}
 For $0<r<1$ consider function $K_{t}(z,w)=\frac{1}{1-r\bar{z}w}G_{t}(w)G_{t}(rz)^{*}$
defined for $(z,w)$ in $\mathbb{T}\times\mathbb{T}$. Let $F(z,w)=K_{t}(z,w)-K_{s}(z,w)$.
Let $\|F\|_{\infty}$ be essential supremum taken against planar measure
over $(z,w)\in\mathbb{T}\times\mathbb{T}$. 

Then $\|F^{*}\|_{\infty}\leq\sqrt{\frac{n}{1-r^{2}}}\|p_{t}-p_{s}\|$\end{prop}
\begin{proof}
Let $\epsilon>0$. Let $z_{0}\in\mathbb{T}$ be such that $\|F(z_{0},\cdot)\|_{\infty}=\|F\|_{\infty}=:M$.
Let $E_{\epsilon}\subset\mathbb{T}$ be such that $m(E_{\epsilon})>0$
and $\|F(z_{0},w)\|_{B(\mathbb{C}^{n})}>\|F\|_{\infty}-\epsilon$
for $w\in E_{\epsilon}$. By $z$-continuity there exists $B_{\epsilon}(z_{0})\subset\mathbb{T}$
such that \[
\|F(z_{0},w)_{(z,w)\in B_{\epsilon}(z_{0})\times E_{\epsilon}}\|_{B(\mathbb{C}^{n})}>\|F\|_{\infty}-2\epsilon.\]
 By Lusin's theorem\emph{ }there exists compact $L\subseteq\mathbb{T}\times\mathbb{T}$
and continuous function $\tilde{F}$ such that $m(L^{c})<\frac{m(B_{\epsilon}(z_{0})\times E_{\epsilon})}{4}$
and $\tilde{F}_{\big|L}=F_{\big|L}$. Let $\delta>0$ such that when
$|w_{1}-w_{2}|<\delta$ we have $\|\tilde{F}(z_{0},w_{1})-\tilde{F}(z_{0},w_{2})\|_{B(\mathbb{C}^{n})}<\epsilon$.
Since $L^{c}$ has a very small measure $L\cap(B_{\epsilon}(z_{0})\times E_{\epsilon})$
has a positive planar measure and hence has a section $\tilde{E}_{z_{0}}$
with a positive measure. In addition $\tilde{E}_{z_{0}}$ can be chosen
so that its diameter $d$ is $d<\delta$. Let $c=(c_{1},c_{2}\ldots c_{n})\in\mathbb{C}^{n}$,
$\|c\|=1$ and $w_{0}\in\tilde{E}_{z_{0}}$ be such that $\left\Vert F^{*}(z_{0},w_{0})c^{\top}\right\Vert _{\mathbb{C}^{n}}\geq\|F^{*}\|_{\infty}-2\epsilon$.
Then\[
\left\Vert (p_{t}-p_{s})\left[P_{H_{\mathbb{C}^{n}}^{2}}\frac{\chi_{\tilde{E}_{z_{0}}}}{m(\tilde{E}_{z_{0}})}c^{\top}\right](rz_{0})\right\Vert _{\mathbb{C}^{n}}\geq\]
\[
\geq\left\Vert \frac{1}{m(\tilde{E}_{z_{0}})}\int_{\tilde{E}_{z_{0}}}F^{*}c^{\top}dw\right\Vert _{\mathbb{C}^{n}}-\left\Vert \frac{1}{m(\tilde{E}_{z_{0}})}\int_{\tilde{E}_{z_{0}}}F^{*}P_{H_{\mathbb{C}^{n}}^{2}}^{\perp}c^{\top}dw\right\Vert _{\mathbb{C}^{n}}\geq\]
\[
\geq\left\Vert F^{*}(z_{0},w_{0})c^{\top}\right\Vert _{\mathbb{C}^{n}}-\left\Vert \frac{1}{m(\tilde{E}_{z_{0}})}\int_{\tilde{E}_{z_{0}}}\big(F^{*}(z_{0},w)-F^{*}(z_{0},w_{0})\big)c^{\top}dw\right\Vert _{\mathbb{C}^{n}}\geq\]
\[
\geq\|F^{*}\|_{\infty}-2\epsilon-\sup_{w\in\tilde{E}_{z_{0}}}\|\tilde{F}(z_{0},w)-\tilde{F}(z_{0},w_{0})\|_{B(\mathbb{C}^{n})}\geq\|F\|_{\infty}-3\epsilon\]
\\
On the other hand, let $k:=\left(k_{z},k_{z}\ldots k_{z}\right)^{\perp}$\[
\left\Vert (p_{t}-p_{s})\left[P_{H_{\mathbb{C}^{n}}^{2}}\frac{\chi_{\tilde{E}}}{m(\tilde{E})}c^{\top}\right](rz_{0})\right\Vert _{\mathbb{C}^{n}}=\]
\[
\left\Vert \int_{\mathbb{T}}\left\langle k;(p_{t}-p_{s})\left[P_{H_{\mathbb{C}^{n}}^{2}}\frac{\chi_{\tilde{E}}}{m(\tilde{E})}c^{\top}\right]\right\rangle dw\right\Vert _{\mathbb{C}^{n}}\leq\]
\[
\left\Vert k\right\Vert _{H_{\mathbb{C}^{n}}^{2}}\cdot\left\Vert (p_{t}-p_{s})\left[P_{H_{\mathbb{C}^{n}}^{2}}\frac{\chi_{\tilde{E}}}{m(\tilde{E})}c^{\top}\right]\right\Vert _{H_{\mathbb{C}^{n}}^{2}}\leq\sqrt{\frac{n}{1-r^{2}}}\cdot\|p_{t}-p_{s}\|\]

\end{proof}

\subsection{Main result}

At this point we are ready to go back to the main result, theorem
\emph{\ref{thm:main result}}, which we copy for convenience.
\begin{thm*}
Let $H^{2}(\mathbb{C}^{n})$ be Hardy space for some finite $n$ and
let $\{p_{t}\}_{t\in[0,1]}$ be a family of orthogonal projections
such that:

1. Family is continuous in norm, $\lim_{t\rightarrow\tau}\|p_{t}-p_{\tau}\|=0$;

2. Each projection $p_{t}H^{2}(\mathbb{C}^{n})$ produces an invariant
subspace of a unilateral shift operator (of multiplicity $n$). 

Then there exists integer $m\leq n$ and a choice of a family of inner
functions $G_{t}\in H^{\infty}(\mathbb{T},B(\mathbb{C}^{m},\,\mathbb{C}^{n}))$
for $t\in[0,1]$ such that:

1. $p_{t}H_{\mathbb{C}^{n}}^{2}=G_{t}H_{\mathbb{C}^{m}}^{2}$; and

2. $\{G_{t}\}_{t\in[0,1]}$ can be chosen to be sup-norm continuous,\linebreak{}
$\textrm{esssup}_{\theta}\|G_{t}(e^{i\theta})-G_{\tau}(e^{i\theta})\|_{B(\mathbb{C}^{m};\mathbb{C}^{n})}\stackrel{t\rightarrow\tau}{=}0$\end{thm*}
\begin{proof}
We pick arbitrary $s\in[0,1]$. Inner function $G_{s}(e^{i\theta})$
is an isometry on the circle for $\theta$ a.e., it is an almost everywhere
strong limit of $G_{s}(re^{i\theta})$ as $r\rightarrow1$. Since
domain has finite dimension $m<\infty$, there must exist $\eta>0$
and a point $r_{\eta}e^{i\theta_{\eta}}\in D$ inside the disk such
that $\ker(G_{s}(r_{\eta}e^{i\theta_{\eta}}))=0$ and we have $\left\Vert G_{s}(r_{\eta}e^{i\theta_{\eta}})x\right\Vert \geq\eta$
for all $x\in\mathbb{C}^{m}$, $\|x\|=1$, the norm of a constant
matrix $G_{s}(r_{\eta}e^{i\theta_{\eta}})$ is bounded from below.
$L^{2}$ norm continuity, established in section 2 gives us uniform
continuity of $G_{t}(r_{\eta}e^{i\theta_{\eta}})$ in $t$ inside
the unit disk for any radius $r<1$. That is, lower bound $\eta$
over the norm of $G_{t}(r_{\eta}e^{i\theta_{\eta}})$ will hold in
some $t$-neighborhood of $s$. With this we have

\begin{eqnarray*}
 & \sqrt{\frac{n}{1-r_{\eta}^{2}}}\|p_{t}-p_{s}\| & \geq\|F^{*}\|_{\infty}\\
 & = & \sup_{z\in\mathbb{T}}\sup_{w\in\mathbb{T}}\left\Vert \frac{1}{1-r_{\eta}\bar{z}w}\big(G_{t}(r_{\eta}z)G_{t}(w)^{*}-G_{s}(r_{\eta}z)G_{s}(w)^{*}\big)\right\Vert _{B(\mathbb{C}^{n})}\\
 & \geq & \frac{1}{1+r_{\eta}}\sup_{w\in\mathbb{T}}\left\Vert \big(G_{t}(r_{\eta}e^{i\theta_{\eta}})G_{t}(w)^{*}-G_{s}(r_{\eta}e^{i\theta_{\eta}})G_{s}(w)^{*}\big)\right\Vert _{B(\mathbb{C}^{n})}\\
 & \geq & \frac{1}{1+r_{\eta}}\sup_{w\in\mathbb{T}}\left\Vert G_{t}(r_{\eta}e^{i\theta_{\eta}})\big(G_{t}(w)^{*}-G_{s}(w)^{*}\big)\right\Vert _{B(\mathbb{C}^{n})}\\
 &  & -\frac{1}{1+r_{\eta}}\sup_{w\in\mathbb{T}}\left\Vert \big(G_{t}(r_{\eta}e^{i\theta_{\eta}})-G_{s}(r_{\eta}e^{i\theta_{\eta}})\big)G_{s}(w)^{*}\right\Vert _{\infty,w\in\mathbb{T}}\\
 & \geq & \frac{\eta}{1+r_{\eta}}\sup_{w\in\mathbb{T}}\left\Vert G_{t}(w)^{*}-G_{s}(w)^{*}\right\Vert _{B(\mathbb{C}^{n})}\\
 &  & -\frac{1}{1+r_{\eta}}\left\Vert G_{t}(r_{\eta}e^{i\theta_{\eta}})-G_{s}(r_{\eta}e^{i\theta_{\eta}})\right\Vert _{B(\mathbb{C}^{n})}\end{eqnarray*}
\linebreak{}
In the last expression the second term $\left\Vert G_{t}(r_{\eta}e^{i\theta_{\eta}})-G_{s}(r_{\eta}e^{i\theta_{\eta}})\right\Vert _{\infty,w\in\mathbb{T}}$
is inside the disk, where we have uniform convergence for $r<1$,
eventually it is negligible. With the remaining first term we get
the desired control from norm continuity of $p_{t}$ to sup norm continuity
of $G_{t}$, $\textrm{esssup}_{w\in\mathbb{T}}\left\Vert G_{t}(w)^{*}-G_{s}(w)^{*}\right\Vert _{B(\mathbb{C}^{n})}\stackrel{t\rightarrow s}{\longrightarrow}0$. 
\end{proof}
\newpage{}\bibliographystyle{amsplain}
\bibliography{math}

\newpage{}

\thispagestyle{empty}Giorgi Shonia

\address{Ohio University, Lancaster}

\address{1570 Granville Pike}

\address{Lancaster, OH 43130-1037}

\email{shonia@ohio.edu}

(740) 654-6711x214
\end{document}